\documentclass{article}
\usepackage[utf8]{inputenc}
\usepackage{amsmath, amsthm, amsfonts, color, soul,setspace, geometry}
\usepackage{graphicx}
\usepackage{setspace}
\newtheorem{theorem}{Theorem}[section]
\newtheorem{lemma}[theorem]{Lemma}
\newtheorem{proposition}[theorem]{Proposition}
\newtheorem{corollary}[theorem]{Corollary}

\newtheorem{definition}[theorem]{Definition}

\newcommand{\fl}[2]{\left\lfloor\frac{#1}{#2}\right\rfloor} 
\newcommand{\Sb}[1]{Sb_#1} 

\newcommand{\kSb}{$k\,$-synchronous bondage} 
\newcommand{\kSbn}{$k\,$-synchronous bondage number} 

\newcommand{\twoSbn}{$2$-synchronous bondage number}
\newcommand{\twoSb}{$2$-synchronous bondage}
\newcommand{\BG}[1]{\text{BG}(#1)}
\newcommand{\MBG}[1]{\text{MBG}(#1)}

\title{Generalized Bondage Number: The \kSbn\ of a Graph}
\author{Anaya, Rey \\[5pt]
Belmonte, Alvaro \\[5pt]
Shank, Nate \\[5pt]
Sinani, Elise \\[5pt]
Walker, Bryan}
\date{\today}

\begin{document}
\maketitle
\begin{abstract}
 We investigate a generalization of the bondage number of a graph called the \textit{\kSbn}. The \kSbn\ of a graph is the smallest number of edges that, when removed, increases the dominating number by $k$.  In this paper, we discuss the \twoSbn\ and then generalize to \kSbn.  We present \kSbn\ for several graph classes and give bounds for general graphs. We propose this characteristic as a metric of the connectivity of a simple graph with possible uses in the field of network design and optimization.
\end{abstract}


\section{Introduction}
\indent Graphs serve as a mathematical tool for analyzing networks where the vertices of graphs can represent stations, transmitters, people, computers, cell phones, or cities while the edges demonstrate the connections between these objects, such as railroads, power lines, friendships, computer connections, signals, or roads. Representing such networks as graphs allows us to apply the tools and properties of graph theory to real world problems thereby providing rigorously proven solutions. 

An unreliable network is one that can be easily disrupted either maliciously or accidentally. To ensure the reliability of a network, an understanding of its purpose and sensitivities must be taken into account throughout its construction and preservation. Thus, we must determine the minimum requirements in order for a network to remain operational as well as which parts of the network can break down and thus cause a failure state. 

When modeling networks with graphs, we can analyze different failure states and, therefore, quantify network reliability. For instance, if the network is operational as long as the graph is connected, then the graph is in a failure state once there are two or more components. Then we can evaluate the network's strength in comparison to other potential network designs by counting the minimum number of vertices which must be removed in order to render the network inoperable. This measure of connectivity is classically defined in \cite{beineke1967connectivity} as the minimum number of vertices whose removal results in a disconnected graph. Other considerations that built upon Harary's foundation included \textit{component order vertex connectivity} and \textit{component order edge connectivity} (see \cite{Gross2006}, \cite{Gross2013}, and \cite{Gross2015}). In the former, network failure occurs when the elimination of vertices causes every component of the graph to have an order less than a given positive integer $k$. The latter is the same in every respect except that we remove edges not vertices.

In this paper, we consider edge removal and its impact on the dominating number of a graph. The minimum number of edge removals that increases the dominating number by one was first considered in \cite{bauer1983domination} and called \textit{dominating line-stability}. Later, in \cite{Fink}, the authors define this same concept to be the \textit{bondage number} of a graph. Here we expand upon the concept of bondage number by investigating edge removals that increase the dominating number by a specified number, $k$. We consider the graph, and hence the network it represents, to be in a failure state when the dominating number increases by $k$. In order to achieve a failure state, we look at the deterministic model; that is, we remove edges intelligently to determine the size of the smallest edge set the removal of which results in a failure state.


\section{Background and Definitions}

\indent For our purposes, all graphs are assumed to be undirected simple graphs; that is, graphs with no loops and in which no two vertices share more than a single edge between them. For other graph theory notation and terminology we will use \cite{West2001}.

Given a graph $G = (V,E)$, a set of vertices $D\subseteq V$ is a \textit{dominating set} if all vertices in $V$ are either in $D$ or adjacent to a vertex in $D$. The \textit{dominating number} of the graph, $\gamma(G)$, is the minimum size of a dominating set in $G$. The study of this characteristic is summarized efficiently in \cite{FundDom}.

In 1990, Fink et al.\@\cite{Fink} introduced, in its modern form, the idea of \textit{bondage number}. The bondage number of a graph, denoted $b(G)$, is the size of the smallest subset of edges of $G$, which, when removed, increases the dominating number. In 2013, Xu \cite{xu2013bondage} defines \textit{bondage set} as $\mathcal{E} \subseteq E$ such that $\gamma(G -\mathcal{E}) > \gamma(G)$. Furthermore a \textit{minimum bondage set} is a bondage set\ with the smallest cardinality. We also know if $\mathcal{E}$ is a minimum bondage set, then $\gamma(G -\mathcal{E}) = \gamma(G) + 1$ since the removal of any single edge can increase the dominating number by at most one.

Given a set of graphs $\mathcal{G}$, we define the \textit{minimum bondage number of $\mathcal{G}$} as $b(\mathcal{G}) = \min\{b(G): G \in \mathcal{G}\}.$ Thus, $b(\mathcal{G})$ is the minimum bondage number over all graphs $G \in \mathcal{G}$. 

Given a graph $G = (V,E)$, we define the \textit{bondage graphs of $G$}, denoted $\BG{G}$, to be the set of all graphs $G' = G-\mathcal{E}$ for some bondage set $\mathcal{E} \subseteq E$. 

Similarly, given a graph $G = (V,E)$, we define the \textit{minimum bondage graphs of $G$}, denoted $\MBG{G}$, to be the set of all graphs $G' = G-\mathcal{E}$ for any minimum bondage set $\mathcal{E} \subseteq E$. 

We will say that $G'$ is the result of a \textit{bondage move} if $G'\in \BG{G}$. So, a bondage move is the removal of a bondage set from $G$ and we will say the \textit{size of the bondage move} is the size of the bondage set. Similarly, we will say that $G'$ is the result of a \textit{minimum bondage move} if $G' \in \MBG{G}$ and we will say the \textit{size of a minimum bondage move} is the size of the minimum bondage set. A minimum bondage move is, therefore, the removal of a minimum bondage set from $G$. 

So, if we are completing iterative bondage moves on some graph $G$, then $b(\MBG{G})$ is the minimum bondage number over the set of graphs resulting from all possible first bondage moves on $G$.

In this paper, we define \textit{\kSb\ set}, \textit{minimum \kSb\ set}, and \textit{\kSbn} which generalize a bondage set, minimum bondage set, and bondage number, respectively, by increasing the dominating number by $k$. Throughout this paper, we will assume that $|V| \geq k+\gamma(G)$ since  the dominating number cannot be larger than the order of the graph.

\begin{definition}[\kSb\ set] 
    Given a graph $G = (V, E)$ and a positive integer $k$, a set $\mathcal{E}\subseteq E$ is a \kSb\ set if $\gamma(G-\mathcal{E}) = \gamma(G) + k$.
\end{definition}

So a set of edges is a \kSb\ set if the removal of the edges increases the dominating number by $k$.

\begin{definition}[minimum \kSb\ set]
    Given a graph $G=(V, E)$, a set $\mathcal{E} \subseteq E$ is a minimum \kSb\ set if $\gamma(G-\mathcal{E}) = \gamma(G) + k$ and for all $E'\subseteq E$ with $|E'|< |\mathcal{E}|$, $\gamma(G-E')<\gamma(G)+k$.
\end{definition}

Therefore, a minimum \kSb\ set is a subset of the edge set whose removal increases the dominating number by $k$ and there does not exist a smaller subset that results in the dominating number increasing by $k$.

\begin{definition}[\kSbn]
    Given a graph $G=(V, E)$ and a positive integer $k$, the \kSbn\ of $G$, denoted $\Sb{k}(G)$, is the size of a minimum \kSb\ set.
\end{definition}

Thus, the \kSbn\ of $G$ is the minimum number of edges that can be removed so that the dominating number increases by $k$. 

Connecting the previous definitions with bondage number, we see that the study of $\Sb{1}$ would be the same as the study of the bondage number, as indicated in the following proposition.

\begin{proposition}
    For any graph $G$, $b(G) = \Sb{1}(G)$.
\end{proposition}

However, the \kSbn\ differs from previous studies on $b(G)$ whenever $k \geq 1$. Our approach is novel in that we can specify any desired increase by setting the value of $k$.

In this paper, we consider \twoSb\ in section \ref{SectionSb2} where we analyze the relationship between \twoSb\ and consecutive bondage moves. We also provide some bounds specific to \twoSb\ for general graphs.  In section \ref{SectionSbk}, we proceed to demonstrate several properties of \kSbn s and provide proofs for the \kSb\ of paths, cycles, trees and complete graphs. We conclude by presenting additional areas of potential research into the value of \kSbn s.


\section{Properties of $\Sb{2}$}
\label{SectionSb2}
The value of studying $\Sb{k}$ is two-fold. First, suppose we obtain $\Sb{k}(G)$ for some graph $G=(V,E)$ by iterative minimum bondage moves. This method will require us to analyze graphs outside of the families typically considered when studying $b(G)$ which adds a level of complexity. Beyond this, however, we find that for some graphs, $G$, $\Sb{k}(G)$ is less than the summation of the size of $k$ iterative minimum bondage moves. For example, consider graph $H$ in Figure \ref{Shank}. Notice $\gamma(H)=5$ and a minimum dominating set is $D=\{c, f, g, i, k\}$. Now, $\Sb{2}(H) = 4$ by removing the edges of the cycle; this removal results in vertices $j, k,$ and $l$ being included in every minimum bondage set as they are isolated vertices and the remaining connected component of $H$ has a dominating number of 4 and a minimum dominating set of $D' = \{c,f,g,i\}$. However, $b(H) = 2$ by removing edges $ac$ and $bd$, and $b(\MBG{H})=3$.  For example, $b(G) = \{ac, bd\}$ resulting in $G'$ and $b(G') = \{ad,ab,bc\}$. Thus, two successive minimum bondage moves requires the removal of 5 edges. 

\begin{figure}[htb]
     \centering
     \includegraphics[scale = .6]{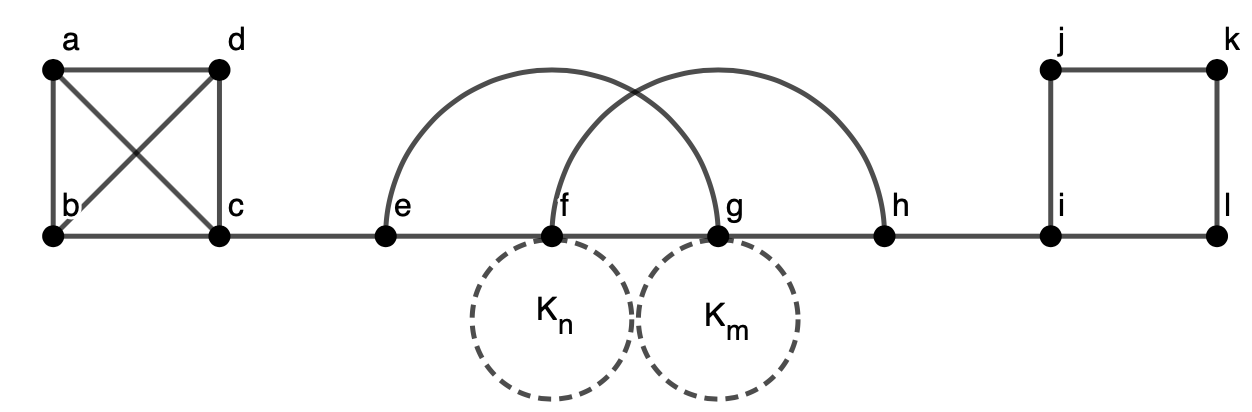}
     \caption{Graph $H$ with $n, m \geq 4$.}
     \label{Shank}
\end{figure}
Alternatively, we can show that there are circumstances under which the sum of the sizes of two successive minimum bondage moves is equal to $\Sb{2}$. 

\begin{theorem}\label{stepwise}
    Let $G$ be a graph. If $b(\MBG{G}) \leq 2$ then $\Sb{2}(G) = b(G) + b(\MBG{G})$.
\end{theorem} 

\begin{proof}
     Since the size of two successive minimum bondage moves serves as an upper bound for $\Sb{2}(G)$, we will assume for the sake of contradiction that there exists a graph $G' \in \MBG{G}$ so that $b(G') = b(\MBG{G}) \leq 2$ and $\Sb{2}(G) < b(G) + b(G')$. 
     Since $\Sb{2}(G) \neq b(G) + b(G')$, there exists a graph $G'' \in \BG{G}-\MBG{G}$ where $G''$ is the result of a bondage move on $G$ of size $y$. Therefore, $\Sb{2}(G)=y + b(G'')$. We know that $y\geq b(G)+1$ so $\Sb{2}(G)\geq b(G)+ 1 + b(G'')$. Hence, 
     
     $$b(G) + 1 + b({G''}) \leq \Sb{2}(G) < b(G) + b(G'),$$ 
     which implies $$1 + b({G''}) < b(G').$$
     
     This yields a contradiction since $b(G'') \geq 1$, but we assumed that $b(G') \leq 2$. 
\end{proof}

Theorem \ref{stepwise} together with the bondage number of a path graph from \cite{Fink} as provided below, assists in determining \twoSb\ of path graphs.
 \begin{theorem}\label{pathbondage}
     The bondage number of a path of order $n \geq 2$ is given by
     \[
     b(P_n) = 
     \begin{cases}
     2, \qquad& \textit{if } n = 1(\operatorname{mod}3) \\
     1, \qquad& \text{else}.
     \end{cases}\]
\end{theorem}

Next, we must show that a pendant edge in $P_n$ is always in a minimum bondage set.
 
\begin{lemma}\label{leftmost}
    For any path graph $P_n$ of order $n\geq 2$ with pendant edge $e$, there exists a minimum bondage set which contains $e$. 
\end{lemma}
 
\begin{proof}
     Given any path graph, $P_n$, of order $n\geq 2$, $\gamma(P_n) = \lceil\frac{n}{3}\rceil$ (see \cite{Fink} for example). Let $e_1$ be a pendant edge in $P_n$. We proceed by cases in using Theorem \ref{pathbondage} to determine $b(P_n)$. 
     
     \textit{\textbf{Case 1}.\textbf{ $n \equiv 0(\operatorname{mod} \:3)$}}: Removing $e_1$ results in $P_1 \oplus P_{n-1}$. The dominating number of the resulting disjoint graph is $\gamma(P_1\oplus P_{n-1}) = \gamma(P_1) + \gamma(P_{n-1})=1+\lceil\frac{n-1}{3}\rceil=1+\frac{n-1+1}{3}=1+\frac{n}{3}=1+\gamma(P_n)$. Since $b(P_n) = 1$ and $\gamma(P_n - \{e_1\}) = 1+\gamma(P_n)$, the set $\{e_1\}$ is a minimum bondage set.
     
     \textit{\textbf{Case 2}.\textbf{ $n \equiv 2(\operatorname{mod} \:3)$}}: Removing $e_1$ results in $P_1 \oplus P_{n - 1}$ whose dominating number is $\gamma(P_1\oplus P_{n-1}) = \gamma(P_1) + \gamma(P_{n-1})=1+\lceil\frac{n-1}{3}\rceil=1+\frac{n-1+2}{3}=1+\frac{n+1}{3}=1+\gamma(P_n)$. Since $b(P_n) = 1$ and $\gamma(P_n - \{e_1\}) = 1+\gamma(P_n)$, the set $\{e_1\}$ is a minimum bondage set.
     
     \textit{\textbf{Case 3. $n \equiv 1(\operatorname{mod} \:3)$}}: Note that $P_n$ with $n>2$ will have two pendant edges, denoted $e_1\text{ and }e_2$. Removing $e_1$ and $e_2$ results in $P_1\oplus P_1\oplus P_{n-2}$. The resulting dominating number of this disjoint graph is $\gamma(P_1\oplus P_1\oplus P_{n-2}) = 2\gamma(P_1) + \gamma(P_{n-2})=2+\lceil\frac{n-2}{3}\rceil=1+\frac{n+2}{3}=1+\gamma(P_n)$. Since $b(P_n) = 2$ and $\gamma(P_n -\{e_1,e_2\}) = 1+\gamma(P_n)$, the set $\{e_1,e_2\}$ is a minimum bondage set.
\end{proof}
 
We now prove \twoSb\ for all $P_n$.
 
\begin{theorem}\label{paths}
     For a path graph, $P_n$,
     $$\Sb{2}(P_n)= \begin{cases}
     2, \: & n\equiv 0(\operatorname{mod} \: 3) \\
     3, \: & \text{else.}
     \end{cases}$$
\end{theorem}
\begin{proof}
     We know from Lemma \ref{leftmost} that each pendant edge in $P_n$ is contained in a minimum bondage set. 
     
     \textit{\textbf{Case 1.} $n \equiv 0(\operatorname{mod} \:3)$.} Consider removing the two pendant edges. Let $P_n' = P_1 \oplus P_1 \oplus P_{n-2}$. So $\gamma(P_n') = 2+ \lceil\frac{n-2}{3}\rceil = 2 +\frac{n}{3}= 2+ \gamma{P_n}.$ So $\Sb{2}(P_n) = 2$. 
    
     \textit{\textbf{Case 2.} $n \equiv 1(\operatorname{mod} \:3)$.} Then $b(P_n) = 2$ by Theorem \ref{pathbondage}. Removing both pendant edges of $P_n$ will result in a disconnected graph $P_n' = (P_1 \oplus P_1 \oplus P_{n-2})$ with $\gamma(P_n') = \gamma(P_n) + 1$ by Lemma \ref{leftmost}. Note that $n - 2 \equiv 2(\operatorname{mod} \: 3)$ so $b(P_n') = b(P_{n-2}) = 1$. Thus, removing a pendant edge from $P_{n-2}$ results in the graph $P_n^{''}$ where $\gamma(P_n^{''}) = \gamma(P_n) + 2$ by Lemma \ref{leftmost}. Then, by Theorem \ref{stepwise}, we know $\Sb{2}(P_n) = 3$ for $n \equiv 1(\operatorname{mod})$.
     
     \textit{\textbf{Case 3.} $n \equiv 2(\operatorname{mod} \:3)$.} Then $b(P_n) = 1$ by Theorem \ref{pathbondage}. Removing a pendant edge from $(P_n)$ will result in a disconnected graph $P_n' = (P_1 \oplus P_{n - 1})$ with $\gamma(P_n') = \gamma(P_n) + 1$ by Lemma \ref{leftmost}. Note that $n - 1 \equiv 1(\operatorname{mod} \: 3)$ so $b(P_n^{'}) = b(P_{n - 1}) = 2$ by Theorem \ref{pathbondage}. Thus, removing the two pendant edges results in the graph $P_n^{''}$ where $\gamma(P_n^{''}) = \gamma(P_n) + 2$ by Lemma \ref{leftmost}. Thus, by Theorem \ref{stepwise}, $\Sb{2}(P_n) = 3$ for $n \equiv 2(\operatorname{mod} \: 3)$.
\end{proof}

To conclude this section, we find bounds for general graphs based on specific characteristics including the degree of a vertex and induced subgraph structure. In \cite{bauer1983domination} and \cite{Fink}, the authors show that the bondage number of a graph is bounded above by the minimum of one less than the sum of the degrees of two adjacent vertices. We can generalize their results to find an upper bound for $\Sb{2}(G)$ based on the degree of several vertices. 

\begin{theorem}\label{InducedP3}
     Let $G(V,E)$ be a graph. Then 
     $$\Sb{2}(G) \leq min\{\operatorname{deg}(u) + \operatorname{deg}(v) + \operatorname{deg}(w) - \sigma(u,v,w)\},$$
     where the minimum is over all sets $\{u,v,w\} \subseteq V$ where $v$ is adjacent to both $u$ and $w$, and $\sigma(u,v,w)$ is the size of the induced subgraph on $\{u,v,w\}$. 
\end{theorem}
\begin{proof}
     Let $\{u,v,w\} \subseteq V$ be such that $v$ is adjacent to both $u$ and $w$. Let $\sigma$ denote the size of the induced subgraph on $u,v,w$, and let $\lambda = \operatorname{deg}(u) + \operatorname{deg}(v) + \operatorname{deg}(w) - \sigma$. If $E'$ is the set of edges incident to $u,v,$ or $w$, then $|E'| = \lambda$.
     
     Assume, for the sake of contradiction, that $\Sb{2}(G) > \lambda$. Therefore, if $G' = G - E'$, then $u$, $v$, and $w$ are isolated in $G'$ and $\gamma(G')= \gamma(G)$ or $\gamma(G')= \gamma(G)+1$.
     If $D$ is a minimum dominating set of $G-\{u,v,w\}$, then $D \cup \{u,v,w\}$ is a minimum dominating set for $G'$ and $D \cup \{v\}$ dominates $G$. Therefore $\gamma(G') = |D|+3$ and $\gamma(G) \leq |D| + 1$ which contradicts that $\gamma(G')= \gamma(G)$ or $\gamma(G')= \gamma(G)+1$.
\end{proof}
 
\begin{theorem}\label{InducedP2s}
     Let $G(V,E)$ be a graph. Then 
     \[\Sb{2}(G) \leq min\{\operatorname{deg}(u) + \operatorname{deg}(v) + \operatorname{deg}(s) + \operatorname{deg}(t) - 2\},\]
     where the minimum is over all subsets $\{u, v, s, t\} \subseteq V$ where $uv, st \in E$ and the size of the induced subgraph on $\{u,v,s,t\}$ is two.
\end{theorem}
\begin{proof}
     Let $\{u,v,s,t\} \subseteq V$ where $uv, st \in E$ and the size of the induced subgraph on $\{u,v,s,t\}$ is two.
     Let $\lambda=\operatorname{deg}(u) + \operatorname{deg}(v) + \operatorname{deg}(s) + \operatorname{deg}(t) - 2$. If $E'$ is the set of edges incident to $u,v,s,\text{ or }, t$, then $|E'| = \lambda$.
     
     Assume, for the sake of contradiction, that $\Sb{2} > \lambda$. Therefore, if $G' = G-E'$, then $u, v, w,$ and $t$ are isolated in $G'$ and $\gamma($G'$) = \gamma(G)$ or $\gamma($G'$) = \gamma(G) +1$. 
     If $D$ is a minimum dominating set of $G-\{u,v,s,t\}$, then $D \cup \{u,v,s,t\}$ is a minimum dominating set for $G'$ and $D \cup \{u,s\}$ dominates $G$. Therefore $\gamma(G') = |D|+4$ and $\gamma(G) \leq |D| + 2$ which contradicts that $\gamma(G')= \gamma(G)$ or $\gamma(G')= \gamma(G)+1$.
\end{proof}

\section{Properties of $\Sb{k}$ and Application of $\Sb{k}$ to Graph Families}
\label{SectionSbk}
The combined size of two successive minimum bondage moves serves as an upper bound for $\Sb{2}$, so it is essential to discuss several concepts regarding the bondage number. Note that for any graph $G$ and edge $e$,  $\gamma(G)\leq \gamma(G-e)\leq \gamma(G) + 1$. This immediately implies the following proposition. 
\begin{proposition}\label{lowerboundk}
    For any graph $G$ and any positive integer $k$, $$\Sb{k}(G) \geq k.$$
\end{proposition}

Furthermore, if a graph has multiple components, the \kSbn\ can be found by taking the minimum of the sum of the $l_i$-synchronous bondage number for a subset of $i$ components where $\sum l_i = k$. We note the example where $k=2$ in the following proposition. This can be easily generalized to larger values of $k$. 

\begin{proposition}\label{disjointbondage}
    If a graph $G$ consists of $n$ components $C_1, C_2,\hdots, C_n$ where $1\leq b(C_1) \leq b(C_2) \leq \hdots \leq b(C_n)$, then
    \[\Sb{2}(G) = min\{\Sb{2}(C_1),\Sb{2}(C_2),\hdots,\Sb{2}(C_n), b(C_1) + b(C_2)\}.\]
\end{proposition}

We can also easily find the \kSbn\ of a graph if there are sufficiently many pendant edges. To do so, Lemma \ref{pendantvertminimumgamma} gives a sufficient condition for a vertex to be in a minimum dominating set.

\begin{lemma}\label{pendantvertminimumgamma}
    Let $G = (V,E)$ be a graph with vertices $r,a,b\in V$ so that $ra,rb\in E$ and $\operatorname{deg}(a) = \operatorname{deg}(b) = 1$ (i.e. $ra$ and $rb$ are pendant edges). Then any minimum dominating set must contain $r$. 
\end{lemma} 

\begin{proof} 
    Let $A\subseteq V$ be any minimum dominating set of $G$, and assume for the sake of contradiction that $r\notin A$. Then, it must be true that $a,b\in A$ since these are adjacent to no other vertices. But observe that the set $A^{'}=(A\setminus \{a,b\})\cup \{r\}$ must also be a dominating set and $|A'| < |A|$, contradicting the minimality of $A$. Thus, $r\in A$. 
\end{proof}

If we have a graph that contains a vertex that is incident to more than one pendant edge, we can find its bondage number of such a graph as shown in the following theorem.

\begin{theorem}\label{2pendantbondage}
    Let $G = (V,E)$ be a graph with vertices $r,a,b\in V$ such that $ra,rb\in E$ and $\operatorname{deg}(a) = \operatorname{deg}(b) = 1$ (that is, $ra$ and $rb$ are pendant edges). Then, $Sb_1(G) = 1$. 
\end{theorem}

\begin{proof}
    First, define $G'=G-ra$, and let $A\subseteq V$ be a minimum dominating set of $G$. By Lemma \ref{pendantvertminimumgamma} above, we know that $r\in A$. The set $A^{'}=A\cup \{a\}$ is a dominating set in $G'$, which we claim is a minimum dominating set. To prove this, suppose for the sake of contradiction that there exists some $B^{'}\subseteq V$ so that $|B^{'}|<|A^{'}|=|A| + 1$ which dominates $G'$. Clearly, $a$ must be in $B^{'}$, so if we define $B = B^{'}-\{a\}$, this is a dominating set of $G$ with $|B|<|A|$, which was a minimal dominating set of $G$. This is a contradiction, so we must have that $A^{'}$ is a minimal dominating set of $G'$, so that the bondage number of $G$ must be one.
\end{proof}
 
Applying Theorem \ref{2pendantbondage} repeatedly to a graph $G$ which contains sufficiently many pendants, along with Proposition \ref{lowerboundk} gives the immediate result

\begin{corollary}
     Let $R \subseteq V$ be such that for all $r \in R$, $r$ is incident to at least two pendant edges in $E$. If $A=\{a \in V: d(a) = 1 \text{ and } a $\text{ is adjacent to a vertex in R}\}, then $Sb_{|A|-|R|}(G) = |A|-|R|$
\end{corollary}

Now we move on to prove the \kSb\ for several well-known graph families: paths, cycles, trees, and complete graphs. 

\begin{theorem}\label{\Sb{k}Paths}
    For a path graph, $P_n$,
    
     $$\Sb{k}(P_n)= \begin{cases}
     \lfloor\frac{3k-1}{2}\rfloor, \: & n\equiv 0(\operatorname{mod} \: 3)\\
     \lfloor\frac{3k+1}{2}\rfloor , \: & n\equiv 1(\operatorname{mod} \: 3) \\
     \lceil\frac{3k-1}{2}\rceil, \: & n\equiv 2(\operatorname{mod} \: 3).
     \end{cases}$$
     
\end{theorem}
\begin{proof}
We will proceed by cases.\\
\textit{\textbf{Case 1:} $n \equiv 0(\operatorname{mod} \:3)$.} First, we will show there exists a set of $\left\lfloor \frac{3k-1}{2}\right\rfloor$ edges, $E'$, so that $\gamma(P_n - E') = \gamma(P_n) + k$. Let $E'$ consists of the leftmost $\left\lfloor \frac{3k-1}{2}\right\rfloor$ edges of $P_n$. For simplicity, let $y = \left\lfloor \frac{3k-1}{2}\right\rfloor$ Thus, $\displaystyle P_n - E' = P_{n-y} \oplus\left( \bigoplus_{i=1}^y P_1\right)$.  So the dominating number of $P_n - E'$ is $$\gamma(P_n - E') = \left\lceil \frac{n-y}{3} \right\rceil + y.$$

Consider that when $k$ is odd, $y = \frac{3k-1}{2}$.  This implies $\left\lceil \frac{n+2y}{3}\right\rceil = \left\lceil\frac{n}{3} + k - \frac{1}{3}\right\rceil = \frac{n}{3}+k$. And when $k$ is even, $y = \frac{3k-1}{2} - \frac{1}{2} = \frac{3k-2}{2}$. This implies $\left\lceil \frac{n+2y}{3}\right\rceil = \left\lceil\frac{n}{3} + k - \frac{2}{3}\right\rceil = \frac{n}{3}+k$. These equalities imply that $\gamma(P_n-E') = \gamma(P_n)+k$. 

Now we move to show $\lfloor\frac{3k-1}{2}\rfloor$ is the minimum number of edges which increase $\gamma(P_n)$ by $k$. Assume for the sake of contradiction that $Sb_k(P_n) < \big\lfloor\frac{3k-1}{2}\big\rfloor$.  This implies that there is a set $\mathcal{E} \subseteq E$ such that $|\mathcal{E}|< \big\lfloor\frac{3k-1}{2}\big\rfloor$ and $\gamma(P_n - \mathcal{E}) \geq \big\lceil\frac{n}{3}\big\rceil + k = \frac{n}{3} + k$. Note that $\displaystyle P_n-\mathcal{E} = \bigoplus^{|\mathcal{E}|}_{i=1}P_{a_i}$ where $\displaystyle \sum^{|\mathcal{E}|+1}_{i=1}a_i = n$. Then
    \begingroup
    \allowdisplaybreaks
    \begin{align*}
        \gamma(P_n - \mathcal{E}) &= \sum^{|\mathcal{E}|+1}_{i=1}\gamma(P_{a_i}) \\
        &= \sum^{|\mathcal{E}|+1}_{i=1} \bigg\lceil\frac{a_i}{3}\bigg\rceil \\
        &= \sum_{a_i \equiv 0(\operatorname{mod} \:3)} \frac{a_i}{3} + \sum_{a_i \equiv 1(\operatorname{mod} \:3)} \frac{a_i + 2}{3} + \sum_{a_i \equiv 2(\operatorname{mod} \:3)} \frac{a_i + 1}{3} \\
        &= \sum^{|\mathcal{E}|+1}_{i=1} \frac{a_i}{3} +\sum_{a_i \equiv 1(\operatorname{mod} \:3)} \frac{2}{3} + \sum_{a_i \equiv 2(\operatorname{mod} \:3)} \frac{1}{3} \\
        &\leq \frac{n}{3} + \sum_{a_i \equiv 1(\operatorname{mod} \:3)}\frac{2}{3} + \sum_{a_i \equiv 2(\operatorname{mod} \:3)}\frac{2}{3} \\
        &\leq \frac{n}{3} + \frac{2}{3}(|\mathcal{E}| + 1) \\
        &\leq \frac{n}{3} + \frac{2}{3}\bigg(\bigg\lfloor\frac{3k-1}{2}\bigg\rfloor\bigg) \\
        &\leq \frac{n}{3} + \frac{2}{3}\bigg(\frac{3k-1}{2}\bigg) \\
        &= \frac{n}{3} + k - \frac{1}{3}.
    \end{align*}
    But this conflicts with our assumption that $\gamma(P_n - \mathcal{E}) \geq  \frac{n}{3}+k$.  Therefore, we have proven that $\Sb{k}(P_n)= \lfloor\frac{3k-1}{2}\rfloor$ for $n\equiv 0(\operatorname{mod} \: 3)$.
    
    Both \textit{\textbf{Case 2:} $n \equiv 1(\operatorname{mod} \:3)$} and \textit{\textbf{Case 3:} $n \equiv 2(\operatorname{mod} \:3)$} follow a similar argument as above and therefore are omitted.  
\end{proof}
\endgroup
Since removing any single edge from a cycle results in a path and the removal of that one edge does not change the cycle's dominating number, we easily derive the following theorem. 

\begin{theorem}\label{\Sb{k}Cycles}
    For a cycle graph, $C_n$,     
    $$\Sb{k}(C_n)= \begin{cases}
    \lfloor\frac{3k+1}{2}\rfloor+1 , \: & n\equiv 1(\operatorname{mod} \: 3) \\
    \lceil\frac{3k-1}{2}\rceil+1, \: & n\equiv 2(\operatorname{mod} \: 3)\\
    \lfloor\frac{3k-1}{2}\rfloor+1, \: & n\equiv 0(\operatorname{mod} \: 3).
    \end{cases}$$
\end{theorem}

For the bondage number of trees, \cite{Fink} established the following upper bound.
 
\begin{theorem}\label{BondageTrees}
    If $T$ is a nontrivial tree, then $b(T) \leq 2$.
\end{theorem}

We can extend Theorem \ref{BondageTrees} to provide a range of values for $Sb_k$, and these values are sharp.

\begin{corollary}\label{\Sb{k}Trees}
    Given a tree, $T$, with at least $k$ edges, then $k \leq \Sb{k}(T)\leq 2k$ and the bounds are sharp.
\end{corollary}
\begin{proof}
    This follows immediately from  Proposition \ref{lowerboundk} and repeated iterations of Theorem \ref{BondageTrees}. The lower bounds is sharp if we consider a star graph with $k$ edges, whose dominating number is 1. To increase the dominating number by $k$ we must remove all $k$ edges. 
    To show that the upper bound is also sharp, we will define a special spider graph, $S^*_k$, as a rooted tree with $2k+2$ vertices so that the root vertex will have $k$ children of degree 2 and one child of degree 1. Notice Figure \ref{TreeGraph} is $S^*_2$. Notice that $|V(S^*_k)|=2k+2$, $|E(S^*_k)|=2k+1$, and $\gamma(S^*_k) = k+1$. If we want to find $\Sb{k}(S^*_k)$ we must produce a graph that has a dominating number of $2k+1$. For a graph or order $n$ vertices to have a dominating number of $n-1$, the graph must have only 1 edge. Therefore $\Sb{k}(S^*_k) = 2k+1-1=2k$. 
\end{proof}

\begin{figure}[htb]
     \centering
     \includegraphics[scale = .7]{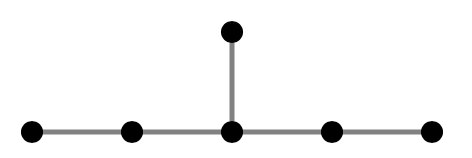}
     \caption{$\Sb{2}(T)= 4.$}
     \label{TreeGraph}
\end{figure}

The result for $\Sb{k}(K_n)$ for complete graphs stems from the following theorem from \cite{Vizing1965} which gives an upper bound for the number of edges in a graph with a specific dominating number.
 
\begin{theorem}\label{Vizing} 
    If $G$ is a graph of order $n$ and $2 \leq \gamma(G) \leq n$, then the number of edges of $G$ is at most $\fl{(n - \gamma(G))(n - \gamma(G) + 2)}{2}$. Equality occurs if and only if $G$ is the disjoint union of $\gamma(G) - 2$ isolated vertices and a graph obtained by removing from an $(n - \gamma(G) + 2)$-clique the edges belonging to a minimum covering.
\end{theorem}
 
Using Theorem \ref{Vizing}, we present the following corollary.

\begin{corollary}\label{\Sb{k}Complete}
    For any complete graph, $K_n$, $$\Sb{k}(K_n)=\binom{n}{2} - \bigg\lfloor\frac{(n-k-1)(n-k+1)}{2}\bigg\rfloor.$$
\end{corollary}
\begin{proof}
    From \cite{Vizing1965}, we know that for any simple graph of order $n$ and dominating number $d$, the maximum number of edges that the graph can have is
    \[\bigg\lfloor\frac{(n-d)(n-d+2)}{2}\bigg\rfloor.\]
    Clearly this graph must be a subgraph of $K_n$. Note also that $K_n$ has $\binom{n}{2}$ edges and $\gamma(K_n)=1$, and so we can remove 
    \[\binom{n}{2} - \bigg\lfloor\frac{(n-k-1)(n-k+1)}{2}\bigg\rfloor\]
    edges to leave only the subgraph mentioned in Theorem \ref{Vizing} with dominating number $k+1$ for any positive integer $k$. The quantity above must be the minimum number of edges we must remove to increase the dominating number of $K_n$ by $k$, because if there were any smaller edge set $E$ whose removal would increase the dominating number by $k$, then the resulting graph $K_n-E$ would contradict Theorem \ref{Vizing}.
    Thus, $$\Sb{k}(K_n)=\binom{n}{2} - \bigg\lfloor\frac{(n-k-1)(n-k+1)}{2}\bigg\rfloor.$$
\end{proof}

 
\section{Conclusion}

In this paper, we have considered applying the idea of $\Sb{k}$ to sparse graphs such as paths, cycles, and trees and the extremely dense complete graph. Investigating $\Sb{k}$ for additional graph families such as grid graphs or $r$-regular graphs would enable a better understanding of real life applications of $\Sb{k}$. The interaction between $\Sb{k}$ and graph operators, including disjoint unions, is another area to consider. 

Furthermore, in order to create sharp bounds without regard to graph families, we must delve into the $G(n,m)$ problem; that is, given any graph on $n$ vertices with $m$ edges, what is the maximum number of edges that we might have to remove in order to cause a failure state? In this way, we avoid being restricted by whether a particular graph belongs to a graph family for which $\Sb{k}$ has been previously determined. If we could state for certain that the removal of $e$ edges would guarantee a failure state, we could turn our focus to efficiently determining which $e$ edges need to be removed.

Our presentation of specific $\Sb{2}$ properties and $\Sb{k}$ in general aims to extend the idea of a bondage number and thereby provide a new criteria for a failure state in a network. Since, in Section 3, we demonstrated that it is possible for a \twoSb\ move to be more effective than the successive one-step counterparts, our proposal cannot be lightly dismissed. To better evaluate the benefits of studying $\Sb{2}$, we would need to establish how much more efficient a \twoSb\ move can be. Similarly, we would like to determine the increase in efficiency achieved by \kSb\ moves when compared to successive $n$-synchronous bondage moves where $n < k$ and $n$ is a factor of $k$. \\


\end{document}